\documentclass[11pt,onecolumn]{article}

\usepackage[a4paper,margin=1.15in]{geometry}
\usepackage[T1]{fontenc}
\usepackage{lmodern}
\usepackage{microtype,booktabs,array}
\usepackage[dvipsnames]{xcolor}
\usepackage{amsmath}
\usepackage{amssymb}
\usepackage{amsthm}
\usepackage{aliascnt}
\usepackage{mathtools}
\usepackage{enumitem}
\usepackage{titlesec}
\usepackage{fancyhdr}
\usepackage[pagebackref=true]{hyperref}
\usepackage[nameinlink,noabbrev]{cleveref}

\definecolor{linkblue}{HTML}{1F4E79}
\hypersetup{
    colorlinks=true,
    linkcolor=linkblue,
    citecolor=linkblue,
    urlcolor=linkblue,
    pdfborder={0 0 0}
}

\titleformat{\section}
    {\large\bfseries\color{linkblue}}{\thesection}{0.75em}{}
\titleformat{\subsection}
    {\normalsize\bfseries\color{linkblue}}{\thesubsection}{0.75em}{}
\titlespacing*{\section}{0pt}{2.3ex plus 0.8ex minus 0.2ex}{1.1ex plus 0.2ex}
\titlespacing*{\subsection}{0pt}{1.8ex plus 0.6ex minus 0.2ex}{0.8ex plus 0.2ex}

\setlist[itemize]{leftmargin=2.2em,itemsep=0.25ex,topsep=0.6ex}
\setlist[enumerate]{leftmargin=2.2em,itemsep=0.25ex,topsep=0.6ex}
\renewcommand{\headrulewidth}{0.4pt}
\renewcommand{\headrule}{\hbox to\headwidth{\color{linkblue}\leaders\hrule height \headrulewidth\hfill}}
\fancypagestyle{plain}{%
    \fancyhf{}
    \fancyfoot[C]{\small\color{linkblue}\thepage}
    \renewcommand{\headrulewidth}{0pt}
}

\renewcommand*{\backref}[1]{}

\renewcommand*{\backrefalt}[4]{%
  \ifcase #1\relax
  \or\space #2.%
  \else\space #2.%
  \fi
}

\newtheoremstyle{blueplain}
    {0.9ex}{0.9ex}{\itshape}{}%
    {\bfseries\color{linkblue}}{.}{0.55em}{}
\newtheoremstyle{bluedefinition}
    {0.9ex}{0.9ex}{\normalfont}{}%
    {\bfseries\color{linkblue}}{.}{0.55em}{}
\theoremstyle{blueplain}
\newtheorem{theorem}{Theorem}[section]
\newaliascnt{proposition}{theorem}
\newtheorem{proposition}[proposition]{Proposition}
\aliascntresetthe{proposition}
\newaliascnt{lemma}{theorem}
\newtheorem{lemma}[lemma]{Lemma}
\aliascntresetthe{lemma}
\newaliascnt{corollary}{theorem}
\newtheorem{corollary}[corollary]{Corollary}
\aliascntresetthe{corollary}
\newaliascnt{conjecture}{theorem}

\aliascntresetthe{conjecture}
\theoremstyle{bluedefinition}
\newaliascnt{definition}{theorem}
\newtheorem{definition}[definition]{Definition}
\aliascntresetthe{definition}
\newaliascnt{example}{theorem}
\newtheorem{example}[example]{Example}
\aliascntresetthe{example}
\theoremstyle{bluedefinition}
\newaliascnt{remark}{theorem}
\newtheorem{remark}[remark]{Remark}
\aliascntresetthe{remark}

\newcommand{\R}{\mathbb R}
\newcommand{\C}{\mathbb C}
\newcommand{\M}{\mathrm{Mat}}
\newcommand{\Sym}{\operatorname{Sym}}
\newcommand{\tr}{\operatorname{tr}}
\newcommand{\Tr}{\operatorname{Tr}}

\title{Clifford Realizations of Hyperbolic Cubics}
\author{\textsc{Tim Netzer}\\[0.25em]
  {\small Department of Mathematics, University of Innsbruck}}
\date{September 4, 2026}

\makeatletter
\renewcommand{\maketitle}{%
    \thispagestyle{plain}%
    \begin{center}
        \vspace*{-2.6em}
        {\LARGE\bfseries\color{linkblue}\@title\par}
        \vspace{1.25em}
        {\large\@author\par}
        \vspace{0.7em}
        {\small\color{gray}\@date\par}
    \end{center}
    \vspace{1.1em}
}
\makeatother

\renewenvironment{abstract}{%
    \begin{center}
    \begin{minipage}{0.9\textwidth}
    \small
    \begin{center}
        {\bfseries\color{linkblue}\abstractname}
    \end{center}
    \vspace{-0.4em}
}{%
    \end{minipage}
    \end{center}
    \vspace{0.8em}
}

\begin{document}
\maketitle

\begin{abstract}
We prove the generalized Lax conjecture for cubics in five variables:
their hyperbolicity cones are spectrahedral.  Our approach constructs
spectrahedral realizations from completely positive maps on Clifford
algebras.  The construction passes through quaternionic determinantal
representations of certain extremal cubics.  Convexity of the set of Clifford-realizable cubics
then yields the result for all cubics in five variables.
The approach also covers all previously known cubic cases of
the generalized Lax conjecture: forms in at most four variables and
symmetric forms in arbitrarily many variables.
We further establish a sufficient criterion for Clifford realizability
in arbitrary dimension and use it to obtain a neighborhood of
the origin in the cubic norm ball consisting of Clifford-realizable
cubics.  Finally, we prove spectrahedrality for sparse hyperbolic
quartics in four variables.
\end{abstract}



\section*{Introduction}
\markboth{Introduction}{}

Hyperbolic polynomials give rise to convex cones and hence to a class of
convex optimization problems.  A basic example is the determinant on the
space of complex Hermitian matrices: its hyperbolicity cone
is the positive semidefinite cone.  More generally, if
$L(y)=\sum_{j=1}^n y_jA_j$ is a complex Hermitian  linear pencil with
$L(e)\succ0$, then $\det L$ is hyperbolic with respect to $e$ and
\[
 \Lambda_+(\det L,e)=\{y\in\R^n\mid L(y)\geqslant0\}.
\]
Cones admitting such a description are called \emph{spectrahedral}. 
The generalized Lax conjecture asks whether the converse holds: is every
hyperbolicity cone spectrahedral?  
For general background on spectrahedra, spectrahedral shadows, and convex
algebraic geometry, see~\cite{BlekhermanParriloThomas, NetzerPlaumann}.

It is important that the conjecture concerns the cone, not a particular
polynomial defining it.  Requiring the polynomial itself, or even a
positive power of it, to have a definite determinantal representation
is stronger and fails in general, already for cubics in sufficiently
many variables~\cite{Saunderson}.  There are nevertheless substantial
positive results.  In three variables every hyperbolic polynomial has
a definite determinantal representation~\cite{HeltonVinnikov}.
Quadratic hyperbolicity cones are spectrahedral in any number of
variables~\cite{NetzerThomDeterminantal}.  Other established families include
elementary symmetric polynomials~\cite{BrandenElementary}, derivative
relaxations of the positive semidefinite cone~\cite{KummerSpectral},
and the multivariate matching polynomials of~\cite{Amini}.
A weaker conclusion holds under a smoothness assumption: if every
nonzero boundary point of the cone is a smooth point of the defining
polynomial, then the cone is a spectrahedral shadow, that is, a linear
image of a spectrahedral cone~\cite{NetzerSanyal}.
This has been strengthened to second-order cone representability under
the weaker assumption of Nash-smooth boundary~\cite{ScheidererSOC}.

For cubics, spectrahedrality in four variables follows from definite
determinantal representations of smooth cubic surfaces and a limiting
argument~\cite{BuckleyKosir,Saunderson}.  It is also known for symmetric
hyperbolic cubics in an arbitrary number of variables
\cite{BlekhermanLindbergShu}.
The main result of this paper concerns five variables, without
smoothness or symmetry assumptions:
\emph{Every hyperbolicity cone of a homogeneous cubic in five variables
is spectrahedral.}

We prove this in \Cref{thm:R4H4} through a realization result for normalized
cubics.  After a linear change of variables and removal of directions
on which the polynomial is constant, a hyperbolic cubic can be written as
\[
 h_q(t,x)=t^3-3\|x\|^2t+2q(x),\qquad \|q\|_\infty\leqslant1,
\]
where $\|q\|_\infty$ is the supremum of $|q|$ on the unit sphere.
Write $H_m$ for this norm ball of cubics on $\R^m$.
We call $q$ Clifford realizable if there is a unital completely positive
map $\Theta\colon C_m\to\M_m(\C)$ satisfying
\[
 q(x)=x^*\Theta(c(x))x,
\]
where $C_m$ is the universal Clifford algebra with generators $c_i$, and
$c(x)=\sum_i x_ic_i$.  Denote the set of such cubics by $R_m$. Every Clifford realization gives an explicit Hermitian linear
pencil for the hyperbolicity cone of $h_q$, of size bounded in terms of $m$ alone
(\Cref{t:ucp}). 
Our central assertion is
\[
 R_4=H_4.
\]
We first show that cubics admitting five suitably independent
sphere contact points are dense among the extreme points of $H_4$.
For these cubics we construct a definite quaternionic Hermitian
$3\times3$ determinantal representation, which   we then  convert into a Clifford
realization.  Finally, compactness and convexity of $R_4$ extend the
construction to all of $H_4$.

The realization method also directly gives $R_3=H_3$ (\Cref{c:lowdim}), Clifford realizations for
symmetric hyperbolic cubics (\Cref{t:symmetric}), and a matrix criterion implying
$m^{-1/2}H_m\subseteq R_m$ (\Cref{t:uniform}, \Cref{c:canonical-row}). We also obtain spectrahedrality for hyperbolic quartics in four variables of the form
$t^4-6\|x\|^2t^2+a(x)$ (\Cref{t:quartic}).

However, the scope of our method is not unlimited: we prove that $H_m$ is not a spectrahedral
shadow for $m\geqslant5$, and consequently $R_m\subsetneq H_m$
(\Cref{t:strict}).  This is an obstruction to Clifford realizability,
not a counterexample to the generalized Lax conjecture.

\section{Setup}

\subsection{Normal form and convexity}
A homogeneous polynomial $h$ is hyperbolic with respect to $e$
if $h(e)>0$ and $s\mapsto h(v+se)$ has only real roots for every real $v$.
The closed hyperbolicity cone $\Lambda_+(h,e)$ is the closure of the
component of $\{h>0\}$ containing $e$.  Its convexity and the invariance of
hyperbolicity under moving $e$ within that component were proved
in~\cite{Garding}.

Fix an integer $m\geqslant2$, and let $P_3(\R^m)$ denote the vector space of
homogeneous cubic forms on $\R^m$.  For $q\in P_3(\R^m)$, define
\[
\|q\|_\infty:=\sup_{\|x\|=1}|q(x)|
\] and set
\[
h_q(t,x):=t^3-3\|x\|^2t+2q(x).
\]
The norm criterion for $h_q$ in the following proposition appears already 
in~\cite{Saunderson}.
We include a proof and the reduction from a general hyperbolic cubic to this normal form.

\begin{proposition}\label{p:normal}
Let $h$ be a homogeneous cubic with $h(e)=1$, hyperbolic with respect to $e$.
After sending $e$ to $(1,0,\ldots,0)$ by a linear change of variables and
quotienting the transverse directions on which $h$ is constant, there is
$q\in P_3(\R^m)$ such that
\begin{equation*}
h(t,x)=t^3-3\|x\|^2t+2q(x).
\end{equation*}
with
\[
\|q\|_\infty\leqslant1.
\]
Conversely, for every $q\in P_3(\R^m)$, the polynomial $h_q$ is
hyperbolic with respect to $(1,0,\ldots,0)$ if and only if
\[
\|q\|_\infty\leqslant1.
\]
\end{proposition}

\begin{proof}
Eliminate the quadratic term in $t$ and write
\[
h(t,x)=t^3-3b(x)t+2q(x).
\]
The discriminant condition gives $b\geqslant0$ and
\[
|q(x)|\leqslant b(x)^{3/2}.
\]
If $u\in\ker b$, then $b(x+su)=b(x)$, so the polynomial
$s\mapsto q(x+su)$ is bounded and therefore constant. Thus $h$ descends to
$\R^m/\ker b$, where $b$ is positive definite and can be changed to
$\|x\|^2$.
Conversely, the discriminant of $t\mapsto h_q(t,x)$ is
\[
108\bigl(\|x\|^6-q(x)^2\bigr).
\]
It has three real roots for every $x$ exactly when $|q(x)|\leqslant\|x\|^3$.
\end{proof}

\begin{corollary}\label{p:Hm}
The set
\[
H_m:=\{q\in P_3(\R^m)\mid\|q\|_\infty\leqslant1\}
\]
is compact and convex. It is exactly the set of cubics $q$ for which $h_q$ is
hyperbolic with respect to $e=(1,0)$. For $q\in\mathring H_m$, the roots of $h_q(\,\cdot\,,x)$ are distinct whenever
$x\ne0$; on $\partial H_m$, two coincide exactly in a direction satisfying
$|q(x)|=\|x\|^3$.
\end{corollary}

\begin{proof}
$H_m$ is the closed unit ball of the norm $\|\cdot\|_\infty$, so it is compact
and convex, and the hyperbolicity assertion is \Cref{p:normal}. The discriminant
\[
108\bigl(\|x\|^6-q(x)^2\bigr).
\]
is positive for $x\ne0$ when $q\in\mathring H_m$, and for
$q\in\partial H_m$ it vanishes exactly in the stated directions.
\end{proof}

The same norm balls in the two- and three-dimensional
spaces and their extreme points are studied in~\cite{HildebrandIvanova}, 
a semidefinite representation of $H_3$ is given
in~\cite{Hildebrand}.
These are precedents for both the convex-body viewpoint and the
contact-point arguments used later.

\subsection{Canonical tensor and positive maps}

We now encode the cubic by a symmetric matrix pencil, turning the
norm bound into a positivity condition for a linear map.

For $q\in P_3(\R^m)$ let $B_q$ be the symmetric trilinear form associated with $q$, so that
$$q(x)=B_q(x,x,x),$$ and define the canonical symmetric pencil $T_q$ by
\[
u^TT_q(x)v=B_q(x,u,v).
\]
Equivalently,
\[
T_q(x)=\frac16\nabla^2q(x),\qquad x^TT_q(x)x=q(x).
\]

\begin{proposition}\label{p:contract}
For every homogeneous cubic $q$ and every $x\in\R^m$,
\begin{equation}\label{eq:tensor-bound}
\|T_q(x)\|\leqslant\|q\|_\infty\|x\|.
\end{equation}
In particular, $h_q$ is hyperbolic if and only if
\begin{equation}\label{eq:contractive-pencil}
\|T_q(x)\|\leqslant\|x\|\qquad(x\in\R^m).
\end{equation}
\end{proposition}

\begin{proof}
By the symmetric multilinear norm theorem for real Hilbert
spaces~\cite{Banach}, the multilinear spectral norm of $B_q$ is attained
on a symmetric rank-one tensor; see also~\cite{Friedland} for
the tensor formulation.  Thus
\[
\sup_{\|x_1\|=\|x_2\|=\|x_3\|=1}|B_q(x_1,x_2,x_3)|
=
\sup_{\|x\|=1}|B_q(x,x,x)|
=\|q\|_\infty.
\]
By multilinearity,
\[
|B_q(x,u,v)|
\leqslant\|q\|_\infty\|x\|\|u\|\|v\|.
\]
Taking the supremum over unit $u,v$ proves \eqref{eq:tensor-bound}. The
equivalence in \eqref{eq:contractive-pencil}
follows from Proposition~\ref{p:normal} and $q(x)=x^TT_q(x)x$.
\end{proof}

Choose Hermitian matrices $\gamma_1,\ldots,\gamma_m\in \M_N(\C)$ satisfying
\[
\gamma_i\gamma_j+\gamma_j\gamma_i=2\delta_{ij}I_N,
\]
and put $\gamma(x)=\sum_i x_i\gamma_i$. Define the operator system
\[
S_m:=\operatorname{span}_\C\{I_N,\gamma_1,\ldots,\gamma_m\}.
\]
This is the complex operator system associated with a spin factor;
for spin factors and their positive trace projections, see~\cite{Stormer}.
Since $\gamma(x)^2=\|x\|^2I_N$,
\begin{equation}\label{eq:spin-cone}
aI_N+\gamma(x)\geqslant0
\Longleftrightarrow
a\geqslant\|x\|.
\end{equation}
Write $T_q(x)=\sum_i x_iT_i$ and define the unital $*$-linear map
\[
\phi_q\colon S_m\to \M_m(\C),\qquad
\phi_q(I_N)=I_m,\quad \phi_q(\gamma_i)=T_i.
\]

\begin{proposition}\label{p:spinpositive}
The polynomial $h_q$ is hyperbolic if and only if $\phi_q$ is positive.
\end{proposition}

\begin{proof}
By \eqref{eq:spin-cone}, positivity is equivalent, after taking
$a=\|x\|$ and replacing $x$ by $-x$, to
\[
-\|x\|I_m\leqslant T_q(x)\leqslant\|x\|I_m.
\]
Now apply Proposition~\ref{p:contract}.
\end{proof}

\begin{remark}
Every positive map $\phi\colon S_m\to \M_r(\C)$ has a positive extension to $\M_N(\C)$.
The construction uses the standard positive trace projection onto a spin
factor~\cite{Stormer}; we recall its elementary proof here.
For $\tau=N^{-1}\Tr$ define
\[
P(X):=\tau(X)I_N+\sum_{i=1}^m\tau(\gamma_iX)\gamma_i.
\]
The identities $\tau(\gamma_i)=0$ and
$\tau(\gamma_i\gamma_j)=\delta_{ij}$ show that $P$ fixes $S_m$. If $X\geqslant0$,
write $a=\tau(X)$ and $x_i=\tau(\gamma_iX)$. Then
$|\langle u,x\rangle|=|\tau(\gamma(u)X)|\leqslant a$ for every unit $u$, so
$\|x\|\leqslant a$ and \eqref{eq:spin-cone} gives $P(X)\geqslant0$.
Thus $\phi\circ P$ is the required extension.
\end{remark}

\section{Clifford realizations}\label{s:clifford}

\subsection{Clifford realizations}

Let $C_m$ be the universal finite-dimensional complex $C^*$-algebra generated by
self-adjoint elements $c_1,\ldots,c_m$ satisfying
\[
c_ic_j+c_jc_i=2\delta_{ij}I,\qquad
c(x):=\sum_{i=1}^m x_ic_i.
\]

\begin{definition}\label{d:DCl}
A \emph{Clifford realization} of a cubic $q\in P_3(\mathbb R^m)$ is a unital
completely positive (ucp) map
\[
\Theta\colon C_m\to \M_m(\C)
\]
such that
\[
q(x)=x^*\Theta(c(x))x.
\]
We call $q$ \emph{Clifford realizable} if such a map exists, and write
$R_m$ for the set of all Clifford-realizable cubics.
\end{definition}
In matrix-convex terminology, the tuples $(\Theta(c_1),\ldots,\Theta(c_m))$
form the level-$m$ matrix range of the universal Clifford tuple;
see~\cite{PasserShalitSolel}.  The set $R_m$ is its image
under the linear map sending a tuple to $x^*\Theta(c(x))x$.

\begin{theorem}
\label{t:ucp}\label{p:common}
For every $q\in R_m$ there is a Hermitian linear pencil
$L_q(t,x)$ of size at most $m2^m+1$ satisfying
\[
\Lambda_+(h_q,e)=\{(t,x)\mid L_q(t,x)\geqslant0\}.
\]
In particular, every Clifford realizable cubic has a spectrahedral
hyperbolicity cone.
\end{theorem}

\begin{proof}
By the dilation theorem for completely positive maps~\cite{Stinespring}, choose a
finite-dimensional representation
\[
\Theta(a)=V^*\pi(a)V,
\qquad
V\colon\C^m\to K.
\]
Since $\Theta$ is unital, $V$ is an isometry. Put
$\Gamma(x)=\pi(c(x))$. Then
\[
q(x)=(Vx)^*\Gamma(x)(Vx),
\qquad
\Gamma(x)^2=\|x\|^2I.
\]
Hence $|q(x)|\leqslant\|x\|^3$, so $h_q$ is hyperbolic by
Proposition~\ref{p:normal}.

Put $r=\|x\|$ and define
\begin{equation*}
L(t,x)=
\begin{pmatrix}
\dfrac t3&\sqrt{\dfrac23}\,(Vx)^*\\[2mm]
\sqrt{\dfrac23}\,Vx&tI+\Gamma(x)
\end{pmatrix}.
\end{equation*}
For $t>r$,
\[
(tI+\Gamma(x))^{-1}
=\frac{tI-\Gamma(x)}{t^2-r^2}.
\]
Using the Schur-complement criterion
\cite[Appendix~A.5.5]{BoydVandenberghe}, the Schur complement of the
lower-right block is
\begin{align*}
\frac t3
-\frac23(Vx)^*(tI+\Gamma(x))^{-1}(Vx)
&=\frac t3-\frac23\frac{tr^2-q(x)}{t^2-r^2}\\
&=\frac{h_q(t,x)}{3(t^2-r^2)}.
\end{align*}
Thus
\[
L(t,x)\succ0
\quad\Longleftrightarrow\quad
t>r\text{ and }h_q(t,x)>0.
\]
Let $\rho(x)$ be the largest root of $t\mapsto h_q(t,x)$. Then
$h_q(r,x)\leqslant0$ and
$\partial_th_q=3(t^2-r^2)\geqslant0$ for $t\geqslant r$, so
$\rho(x)\geqslant r$. Taking closures gives
\[
\Lambda_+(h_q,e)=\{(t,x)\mid L(t,x)\geqslant0\}.
\]
The dilation space may be taken to be a quotient of
$C_m\otimes\C^m$, so its dimension is at most $m2^m$.
\end{proof}

\begin{remark}
The pencil constructed in the proof of \Cref{t:ucp} satisfies
\[
 \det L(t,x)=\frac13 h_q(t,x)
 (t^2-\|x\|^2)^{d-1},\qquad 2d\leqslant m2^m.
\]
Indeed, for $x\ne0$, a unit vector $y\perp x$ gives an invertible
$\Gamma(y)$ anticommuting with $\Gamma(x)$, so the eigenvalues
$\pm\|x\|$ of $\Gamma(x)$ have equal multiplicity $d$.
The formula follows from the Schur complement computed above for
$t>\|x\|$, and hence holds identically as a polynomial identity.
Since $\partial_t h_q=3(t^2-\|x\|^2)$, the factor multiplying $h_q$
is a positive scalar multiple of a power of its derivative:
\[
 \det L=3^{-d}h_q(\partial_t h_q)^{d-1}.
\]
\end{remark}

Before constructing further realizations, we study the geometry of $R_m$, which will let us combine realizations and pass to limits.

\begin{proposition}\label{p:clgeom}
The set $R_m$ is a compact, convex, $O(m)$-invariant spectrahedral shadow.
It meets $\partial H_m$.
\end{proposition}

\begin{proof}
The ucp maps form a compact spectrahedron through their Choi matrices
on the matrix-algebra summands of $C_m$~\cite{Choi}, and the map to cubics is linear.
Thus $R_m$ is a compact, convex spectrahedral shadow. For $U\in O(m)$, let
$\alpha_U$ be the $*$-automorphism determined by
\[
\alpha_U(c(x))=c(Ux).
\]
If $\Theta$ generates $q$, then
\[
\Theta_U(a):=U\,\Theta(\alpha_{U^T}(a))U^T
\]
generates $q(U^Tx)$.

To see that the boundary is met, choose a representation $\pi$ and a unit vector
$v$ with $\pi(c_1)v=v$. The Clifford relations give
\[
\langle v,\pi(c_i)v\rangle=0\qquad(i=2,\ldots,m).
\]
Thus the state $\varphi(a)=\langle v,\pi(a)v\rangle$ satisfies
$\varphi(c_1)=1$ and $\varphi(c_i)=0$ for $i>1$, and
\[
\Theta(a)=\varphi(a)I_m
\]
is ucp and generates
\[
q(x)=x_1\|x\|^2,\qquad \|q\|_\infty=1.\qedhere
\]
\end{proof}

The following obstruction reduces to the nonrepresentability theorem
for cones of nonnegative forms in~\cite{Scheiderer}.  It concerns the convex
body $H_m$, not the spectrahedrality of each individual
hyperbolicity cone.

\begin{theorem}\label{t:strict}
For every $m\geqslant5$, the convex body $H_m$ is not a spectrahedral shadow.
Consequently
\[
R_m\subsetneq H_m\qquad(m\geqslant5).
\]
\end{theorem}

\begin{proof}
First take $m=5$ and identify $\R^5$ with
\[
V:=\{A\in\Sym_3(\R)\mid\Tr A=0\}
\]
with its trace inner product. Put
\[
q_0(A):=\sqrt6\,\Tr(A^3),
\qquad
\nu(x):=\frac{3xx^T-\|x\|^2I}{\sqrt6}.
\]
Lagrange multipliers applied to the eigenvalues of $A$ show that
$q_0(A)\leqslant1$ on the unit sphere of $V$, with equality precisely at
$A=\nu(x)$, $x\in S^2$. Also
\[
\|\nu(x)\|=\|x\|^2,
\qquad q_0(\nu(x))=\|x\|^6.
\]
Suppose that $H_5$ is a spectrahedral shadow.  Polarity preserves
semidefinite representability
\cite{GouveiaNetzer}, so its polar and the following affine
section are spectrahedral shadows:
\[
E:=\{T\in H_5^\circ\mid\langle q_0,T\rangle=1\}
=\operatorname{conv}\{\nu(x)^{\otimes3}\mid x\in S^2\}.
\]
Here the polar is taken with respect to the pairing characterized
by $\langle q,A^{\otimes3}\rangle=q(A)$. Consider
\[
U\colon\operatorname{Sym}^3(V^*)\to\R[x_1,x_2,x_3]_6,
\qquad U(q)(x)=q(\nu(x)).
\]
This map is onto. Indeed, sixth powers span the ternary sextics, and for
$u\in S^2$, $B=\nu(u)$, and $a=\Tr(AB)$, the cubic
\[
Q_u(A):=\frac1{27}\left(
8a^3+12\sqrt6\,a\Tr(A^2B)+12\Tr(A^3B)+q_0(A)
\right)
\]
satisfies $Q_u(\nu(x))=(u^Tx)^6$ by direct substitution.

Let $P_{3,6}$ be the cone of nonnegative ternary sextics. Its dual is the
closed conical hull of the point evaluations $\operatorname{ev}_x$. Since
$U$ is onto and
\[
U^*(\operatorname{ev}_x)=\nu(x)^{\otimes3},
\]
$U^*$ identifies $P_{3,6}^{\vee}$ with the conical hull of $E$. But
$P_{3,6}$ is not a spectrahedral shadow
by~\cite{Scheiderer}. Since duals of closed
spectrahedral-shadow cones are spectrahedral shadows, $P_{3,6}^{\vee}$ is not
one either. This contradicts the semidefinite representability of $E$, since
homogenizing a lift of $E$ gives one for its conical hull.

For $m>5$, the cubics in the first five variables form a linear section of
$H_m$ equal to $H_5$. Thus $H_m$ is not a spectrahedral shadow. The set
$R_m$ is a spectrahedral shadow, being a linear image of the
compact spectrahedron of ucp maps. Hence
$R_m\subsetneq H_m$.
\end{proof}

\subsection{Clifford realizations from decomposable maps}

A positive $*$-linear map $\Psi\colon \M_N(\C)\to \M_m(\C)$ is decomposable if
\[
\Psi=\Psi_++\Psi_-\circ t,
\]
where $t$ denotes transpose and $\Psi_+$ and $\Psi_-$ are completely
positive.  The finite-dimensional Kraus representation of
completely positive maps~\cite{Choi} gives the equivalent expression
\begin{equation}\label{eq:decomposable-kraus}
\Psi(X)=\sum_\alpha V_\alpha^*XV_\alpha
+\sum_\beta W_\beta^*X^TW_\beta.
\end{equation}

\begin{proposition}\label{p:decomposable}
Let $\Psi\colon \M_N(\C)\to \M_m(\C)$ be unital and decomposable, and set
\[
q_\Psi(x):=x^*\Psi(\gamma(x))x.
\]
Then $q_\Psi\in R_m$.
\end{proposition}

\begin{proof}
Stack the Kraus operators in \eqref{eq:decomposable-kraus} on top of each other. Unitality gives
an isometry
\[
V\colon \C^m\to(\C^N)^{\oplus a}\oplus(\C^N)^{\oplus b}.
\]
On this direct sum put
\[
\Gamma(x)
=(\gamma(x)\otimes I_a)\oplus(\gamma(x)^T\otimes I_b).
\]
The summands satisfy the Clifford relations, and
\[
(Vx)^*\Gamma(x)(Vx)
=x^*\Psi(\gamma(x))x
=q_\Psi(x).
\]
The universal property of $C_m$ gives a $*$-representation
$\pi\colon C_m\to B(K)$. Hence $\Theta(a)=V^*\pi(a)V$ is ucp and generates
$q_\Psi$.
\end{proof}

\subsection{The case \texorpdfstring{$m=3$}{m=3}}

The next spectrahedrality conclusion also follows from~\cite{BuckleyKosir} and~\cite{Saunderson}.
The representation of the coefficient body $H_3$ in~\cite{Hildebrand}
also uses low-dimensional positive-map methods, but via a different
construction. 

\begin{corollary}\label{c:lowdim}
We have
\[
R_3=H_3.
\]
Consequently, every hyperbolic cubic in four variables has a
spectrahedral hyperbolicity cone.
\end{corollary}

\begin{proof}
The Pauli matrices give $S_3=\M_2(\C)$, and
Proposition~\ref{p:spinpositive} makes
\[
\phi_q\colon \M_2(\C)\to \M_3(\C)
\]
unital and positive. The decomposability theorem in~\cite{Woronowicz}
therefore applies, so Proposition~\ref{p:decomposable} proves the equality.
The cone statement follows from Theorem~\ref{t:ucp}.
\end{proof}

\subsection{Symmetric cubics}

The spectrahedrality of symmetric hyperbolic cubics was proved
in~\cite{BlekhermanLindbergShu}.
The proof below gives a Clifford realization of the normal-form cubic, and hence another proof of spectrahedrality.

\begin{theorem}\label{t:symmetric}
Let $p$ be a homogeneous cubic in $m+1$ variables which is invariant under
permutations of the variables and hyperbolic with respect to
$\mathbf1=(1,\ldots,1)\in\R^{m+1}$. Then its normal-form transverse cubic
belongs to $R_m$. Consequently, the hyperbolicity cone of $p$ is
spectrahedral.
\end{theorem}

\begin{proof}
The case $m\leqslant1$ is immediate. Let $m\geqslant2$ and put
$V=\mathbf1^\perp\subseteq\R^{m+1}$. If $p$ is a cube, its transverse cubic
is zero and belongs to $R_m$: take
$\Theta(a)=\tau(a)I_m$, where $\tau$ is a tracial state on $C_m$. Assume that
$p$ is not a cube.

The invariant quadratic and cubic forms on $V$ are spanned by $\|x\|^2$ and
$\sum_ax_a^3$. The normal form is therefore
\[
h_\lambda(t,x)=t^3-3\|x\|^2t+2\lambda q(x),
\qquad
q(x):=\frac{\sqrt{m(m+1)}}{m-1}
\sum_{a=1}^{m+1}x_a^3,
\qquad x\in V.
\]
Here $\|q\|_\infty=1$, so $|\lambda|\leqslant1$. Since
$R_m$ is convex and invariant under sign, it suffices to prove
$q\in R_m$.

Let $P$ be the orthogonal projection of $\R^{m+1}$ onto $V$, and put
\[
s_a:=\sqrt{\frac{m+1}{m}}\,Pe_a
\quad(a=1,\ldots,m+1).
\]
Then the $s_a$ are unit vectors in $V$ satisfying
\[
\sum_as_a=0,
\qquad
\sum_as_as_a^T=\frac{m+1}{m}I_V,
\qquad
q(x)=\frac{m^2}{m^2-1}\sum_{a=1}^{m+1}\langle s_a,x\rangle^3
\quad(x\in V).
\]
Identify $V$ orthogonally with $\R^m$, and write
$c(v)=\sum_i v_i c_i$ for $v\in V$. Let $\tau_0$ be the canonical trace on
$C_m$; it extracts the scalar coefficient in the Clifford expansion. Equip
$C_m$ with $\langle a,b\rangle=\tau_0(a^*b)$, put
$K:=C_m^{\oplus(m+1)}$, and let $\pi(a)$ be the operator of left multiplication
by $a$ in each summand. Define
\[
J\colon V_\C\to K,
\qquad
(Jx)_a:=\frac{
\bigl(m\langle s_a,x\rangle I-c(x)\bigr)(I+c(s_a))}
{\sqrt{2(m^2-1)}}.
\]
The Clifford relations and trace identities, together with
\[
\sum_a\bigl|\langle s_a,x\rangle\bigr|^2
=\frac{m+1}{m}\|x\|^2
\qquad(x\in V_\C),
\]
give
\[
\|Jx\|^2=\|x\|^2.
\]
Thus $J$ is an isometry. For $x\in V$, the Clifford relations and
$\sum_a\langle s_a,x\rangle=0$ give
\begin{align*}
(Jx)^*\pi(c(x))(Jx)
&=\frac{m^2}{m^2-1}\sum_a\langle s_a,x\rangle^3\\
&=q(x).
\end{align*}
Thus $\Theta(a)=J^*\pi(a)J$ is a Clifford realization of $q$. Now apply
Theorem~\ref{t:ucp}.
\end{proof}

\subsection{A contraction criterion and a neighborhood of the origin}

We next prove a sufficient condition for Clifford realizability by constructing the required ucp map explicitly. For related matrix-range
and dilation results, see~\cite{EvertPasser,PasserShalitSolel}. 

\begin{theorem}\label{t:uniform}
Let $m\geqslant1$ and let $A_1,\ldots,A_m\in\Sym_m(\R)$ satisfy
\[
\sum_{i=1}^m A_i^2\leqslant I_m.
\]
There is a ucp map $\Theta\colon C_m\to\M_m(\C)$ such that
$\Theta(c_i)=A_i$ for every $i$. Consequently, if
\[
q(x)=x^TA(x)x,\qquad A(x):=\sum_{i=1}^m x_iA_i,
\]
then $q\in R_m$ and $\Lambda_+(h_q,e)$ is spectrahedral.
\end{theorem}

\begin{proof}
Let $\tau_0$ be the canonical trace on $C_m$, namely the coefficient of
$I$ in the ordered Clifford-monomial basis.  Let $\ell_a$ denote left
multiplication by $a$ on $L^2(C_m,\tau_0)$. Define
\[
V_0u:=\frac1{\sqrt2}
\left(I\otimes u+\sum_i c_i\otimes A_i u\right).
\]
The Clifford trace identities give
\[
V_0^*V_0=\frac12\left(I_m+\sum_iA_i^2\right).
\]
Put
\[
D:=\left(\frac12\left(I_m-\sum_iA_i^2\right)\right)^{1/2},
\qquad
V_1u:=I\otimes Du,
\]
and set $V=V_0\oplus V_1$. Then $V^*V=I_m$. On two copies of
$L^2(C_m,\tau_0)\otimes\C^m$, define
\[
\pi(a):=(\ell_a\otimes I_m)\oplus(\ell_a\otimes I_m).
\]
Since
\[
\tau_0(c_i)=0,
\qquad
\tau_0(c_ic_j)=\delta_{ij},
\qquad
\tau_0(c_ic_jc_k)=0,
\]
we have
\[
V_0^*(\ell_{c_j}\otimes I_m)V_0=A_j,
\qquad
V_1^*(\ell_{c_j}\otimes I_m)V_1=0.
\]
Thus $V^*\pi(c_j)V=A_j$ for every $j$. The map
$\Theta(a)=V^*\pi(a)V$ is ucp, and
\[
x^*\Theta(c(x))x
=x^*A(x)x
=q(x).
\]
Thus $\Theta$ is a Clifford realization of $q$. Now apply
Theorem~\ref{t:ucp}.
\end{proof}

Applying the matrix criterion to the canonical pencil gives a test
directly in terms of the cubic, which holds on a neighborhood of the origin.

\begin{corollary}\label{c:canonical-row}
Let $q\in P_3(\R^m)$ and write $T_q(x)=\sum_i x_iT_i$. If
\[
\sum_iT_i^2\leqslant I_m,
\]
then $q\in R_m$ and $\Lambda_+(h_q,e)$ is spectrahedral. In particular,
this holds whenever $\|q\|_\infty\leqslant m^{-1/2}$.
\end{corollary}

\begin{proof}
Apply Theorem~\ref{t:uniform} with $A_i=T_i$, using
$q(x)=x^TT_q(x)x$. For the last assertion,
Proposition~\ref{p:contract} gives $\|T_i\|\leqslant\|q\|_\infty$, so
\(
\sum_iT_i^2\leqslant m\|q\|_\infty^2I_m\leqslant I_m.
\)
\end{proof}

The last criterion can also hold at the  boundary of $H_m$, as the following example illustrates.
\begin{example}\label{ex:diagonal-row}
For $q(x)=\sum_i a_i x_i^3$, the canonical coefficients are
$T_i=a_i e_ie_i^T$, whence
\[
\sum_iT_i^2=\operatorname{diag}(a_1^2,\ldots,a_m^2).
\]
Moreover, $\|q\|_\infty=\max_i|a_i|$: the upper bound follows from
$\sum_i|x_i|^3\leqslant1$ on the unit sphere, and equality is attained
at a coordinate vector. Thus Corollary~\ref{c:canonical-row} applies to
every hyperbolic cubic in this family, including its boundary, in every
dimension. By the $O(m)$-invariance of $R_m$, the same conclusion holds
after any orthogonal change of variables.
\end{example}

\subsection{Sparse quartics in 4 variables}

The ingredients of the next construction are the sum-of-squares theorem
for nonnegative ternary quartics~\cite{Hilbert}, and the low-dimensional
decomposability theorem in~\cite{Woronowicz}. 

\begin{theorem}\label{t:quartic}
Let $a\colon \R^3\to\R$ be a homogeneous quartic, and put
\[
p_a(t,x)=t^4-6\|x\|^2t^2+a(x).
\]
Then $p_a$ is hyperbolic with respect to $e=(1,0)\in\R^4$ if and only if
\begin{equation}\label{eq:quartic-bound}
0\leqslant a(x)\leqslant9\|x\|^4\qquad(x\in\R^3).
\end{equation}
Whenever these inequalities hold, the hyperbolicity cone of $p_a$ is
spectrahedral.
\end{theorem}

\begin{proof}
Put $r=\|x\|$ and $s=t^2$. The two roots in $s$ are
\[
s_\pm=3r^2\pm\sqrt{9r^4-a(x)}.
\]
Thus all four roots in $t$ are real exactly when $s_\pm$ are real and
nonnegative, which is equivalent to \eqref{eq:quartic-bound}.

Assume \eqref{eq:quartic-bound} and set
\[
g(x)=9\|x\|^4-a(x).
\]
By the sum-of-squares theorem for nonnegative ternary quartics
\cite{Hilbert},
\[
g=q_1^2+q_2^2+q_3^2
\]
for real quadratic forms $q_j$. With the Pauli matrices, set
\[
Q(x)=\sum_jq_j(x)\sigma_j,\qquad
Q(x)^2=g(x)I_2.
\]
Writing $q_j(x)=x^TA_jx$ and
$\eta(x)=(q_1(x),q_2(x),q_3(x))$, define
\[
\Phi\colon \M_3(\C)\to \M_2(\C),\qquad
\Phi(X)=3\Tr(X)I_2+\sum_j\Tr(A_jX)\sigma_j.
\]
For $z=u+iv$, the Pauli relations give
\[
\lambda_{\min}\Phi(zz^*)
=
3\|z\|^2-\|\eta(u)+\eta(v)\|
\geqslant
3\|z\|^2-\sqrt{g(u)}-\sqrt{g(v)}
\geqslant0,
\]
where the last inequality uses $g(y)\leqslant9\|y\|^4$. Positive semidefinite
matrices are sums of rank-one matrices, so $\Phi$ is positive and thus decomposable by
the decomposability theorem in~\cite{Woronowicz}.   Evaluating a
decomposable Kraus representation \eqref{eq:decomposable-kraus} at $xx^T$
and stacking its rows produces a matrix $W(x)$, linear in $x$, such that
\begin{equation}\label{eq:quartic-gram}
W(x)^*W(x)
=
\Phi(xx^T)
=3\|x\|^2I_2+Q(x).
\end{equation}
After appending zero rows, let $W$ have $N\geqslant2$ rows and put
\[
L_a(t,x)=
\begin{pmatrix}
tI_N&W(x)\\
W(x)^*&tI_2
\end{pmatrix}.
\]
For $t>0$, the Schur complement and \eqref{eq:quartic-gram} give
\begin{align*}
L_a(t,x)\geqslant0
&\Longleftrightarrow
t^2I_2-W(x)^*W(x)\geqslant0
\\
&\Longleftrightarrow
(t^2-3\|x\|^2)I_2-Q(x)\geqslant0
\\
&\Longleftrightarrow
t^2\geqslant3\|x\|^2+\sqrt{g(x)}.
\end{align*}
The last inequality says that $t$ is at least the largest root of
$p_a(\cdot,x)$. For $t<0$ the pencil cannot be positive semidefinite, while
for $t=0$ it is positive semidefinite only if $W(x)=0$, equivalently $x=0$,
because $\Tr(W(x)^*W(x))=6\|x\|^2$. Hence
\[
\Lambda_+(p_a,e)
=
\{(t,x)\mid L_a(t,x)\geqslant0\}.\qedhere
\]
\end{proof}

\section{The case of five variables}\label{sec:five-variables}

We will now prove the main result of this paper, spectrahedrality of cubic hyperbolicity cones in $\R^5.$
Our strategy is to reduce to cubics with sufficiently many spherical contact points, construct quaternionic determinantal representations, turn them into Clifford realizations, and then use closedness and
convexity to recover the whole body $H_4$.

Throughout this section
\[
 \mathcal P=P_3(\mathbb R^4),\qquad \dim\mathcal P=20,
\]
and, for $q\in H_4$, we write
\[
 Z(q):=\{x\in S^3\mid q(x)=1\}.
\]
The points of $Z(q)$ are called the contacts of $q$.  At every
$x\in Z(q)$, homogeneity and maximality on the sphere give
\begin{equation}\label{eq:contact-gradient}
 \nabla q(x)=3x.
\end{equation}
Call a five-tuple $X=(v_a)_{a=1}^5\in(S^3)^5$ \emph{poised} if every
four of the vectors $v_a$ are linearly independent.  A cubic
$q\in H_4$ is a \emph{five-contact cubic} if $Z(q)$ contains a poised
five-tuple; additional contacts are allowed.

\subsection{Denseness of five-contact cubics}

The use of sphere contacts and their derivative constraints to describe
faces of a cubic norm ball has already been used in \cite{HildebrandIvanova} in a lower-dimensional case. 

\begin{lemma}\label{lem:poised}
Let $X=(v_a)_{a=1}^5$ be poised.  Then the gradient evaluation map
\[
 J_X\colon\mathcal P\to(\mathbb R^4)^5,
 \qquad J_X(r)=(\nabla r(v_a))_{a=1}^5,
\]
is an isomorphism.  Consequently, if $X\subseteq Z(q)$ for some $q\in H_4$,
then, for all $\lambda_a>0$, the functional
\[
 \ell(r)=\sum_{a=1}^5\lambda_a r(v_a)
\]
exposes $q$ in $H_4$.
\end{lemma}

\begin{proof}
After an invertible linear change of coordinates, take
\[
 v_a=e_a\quad(1\leqslant a\leqslant4),
 \qquad v_5=u=(u_1,u_2,u_3,u_4).
\]
Poisedness says precisely that every $u_i$ is nonzero.  The cubics satisfying
$\nabla r(e_i)=0$ for $1\leqslant i\leqslant4$ are exactly
\[
 r(x)=\sum_{i=1}^4d_i\prod_{j\ne i}x_j.
\]
On this four-dimensional kernel, the remaining map
$r\mapsto\nabla r(u)$ has determinant
\[
 -3(u_1u_2u_3u_4)^2\neq 0.
\]
It is therefore invertible.  Thus $J_X$ is injective, hence an isomorphism
because its domain and codomain both have dimension twenty.

Now every $r\in H_4$ satisfies
$\ell(r)\leqslant\sum_a\lambda_a=\ell(q)$.  If equality holds, then
$r(v_a)=1$ for every $a$ because all $\lambda_a$ are positive.  Hence
$\nabla r(v_a)=3v_a=\nabla q(v_a)$ by \eqref{eq:contact-gradient}, so
$J_X(r-q)=0$.  The injectivity of $J_X$ gives $r=q$, proving that $\ell$
exposes $q$.
\end{proof}

\begin{lemma}\label{lem:generic-contacts}
There is a dense set of functionals 
\(c\in\mathcal P^*\) whose unique maximizer \(q_c\) from $H_4$ has five distinct
contacts \((v_a)_{a=1}^5\) such that every four vectors
\[
 p_a=(1,v_a)\in\mathbb R^5
\]
are linearly independent.
\end{lemma}

\begin{proof}
For \(x\in S^3\), let \(\delta_x(r)=r(x)\).  Since
\[
 H_4=\{r\in\mathcal P\mid \delta_x(r)\leqslant1
                         \text{ for every }x\in S^3\},
\]
support-function and polar duality
\cite{Rockafellar} identify the functionals maximized
at \(q\) with the cone generated by the active constraints:
\[
 \operatorname{cone}\{\delta_x\mid x\in Z(q)\}.
\]
We will call a contact set bad if it contains no five points whose lifts to $\R^5$ are
four-wise independent.  Four contacts with dependent lifts lie on a circle
\(C\subseteq S^3\).  Under a rational parametrization of \(C\), clearing
denominators turns \(1-q\) into a polynomial of degree at most six with four
double zeros.  Hence \(C\subseteq Z(q)\).

Consequently a finite bad contact set has at most four points.  If a bad
contact set is infinite and its lifts span more than three dimensions,
choose independent lifts \(p_1,\ldots,p_4\).  Every other lift lies in at
least one of the spaces
\[
 L_a=\operatorname{span}\{p_b\mid b\ne a\},\qquad 1\leqslant a\leqslant4;
\]
otherwise the five lifts would be four-wise independent.  Some \(L_a\)
contains infinitely many contact lifts.  The slice
\[
 L_a\cap(\{1\}\times\mathbb R^4)=\{1\}\times A_a
\]
defines an affine plane \(A_a\subseteq\mathbb R^4\).  It meets \(S^3\) in a
circle \(C\), since the intersection of an affine two-plane with a sphere is
empty, a point, or a circle, and here it is infinite.  The preceding argument,
applied to four of these contacts, gives \(C\subseteq Z(q)\).  There
cannot be two contacts \(y,z\notin C\): if \(P\) is the span of the lifted
circle and \(Y=\operatorname{span}(p_y,p_z)\), then
\(\dim(P\cap Y)\leqslant1\), since otherwise \(Y\subseteq P\) and hence
\(y,z\in C\).  Thus three generic circle lifts together with
\(p_y,p_z\) are four-wise independent.  If all lifts span at most three
dimensions, the contact set is itself contained in a circle. 
We have thus shown that every
bad contact set has at most four points or is contained in
\(C\cup\{y\}\).

The functionals maximized at bad points therefore lie in a semialgebraic
set of dimension at most \(17<20=\dim\mathcal P^*\).  Indeed, the spaces \(E_C\) vary semialgebraically with \(C\). The finite case has at most
\(4\cdot3+4=16\) parameters.  Circles form a
six-dimensional family, and for fixed \(C\) the space
\[
 E_C=\operatorname{span}\{\delta_x\mid x\in C\}
\]
has dimension at most seven: after rationally parametrizing \(C\) and
clearing denominators, its restrictions are polynomials of degree at most
six.  The circle case therefore has
dimension at most
\[
 6+3+(7+1)=17.
\]
Almost every functional on the compact convex set \(H_4\) has a unique
maximizer, by almost-everywhere differentiability of its support function
\cite{Rockafellar}.   A lower-dimensional
semialgebraic set has empty interior.  Removing the bad functionals
therefore leaves a dense set with the required property.
\end{proof}

\begin{proposition}\label{prop:density-five}
The five-contact cubics are dense in
$\operatorname{Ext}(H_4)$.
\end{proposition}

\begin{proof}
Take a functional from Lemma~\ref{lem:generic-contacts}, with unique
maximizer $q$ and contacts $(v_a)_{a=1}^5$.  By
\eqref{eq:contact-gradient}, the lifts $p_a=(1,v_a)$ are singular points of
$h_q$.  For an interior direction $e'$ close to $e=(1,0)$, rescaled so that
$h_q(e')=1$, hyperbolicity with respect to $e'$ follows from the
invariance of hyperbolicity directions inside the cone
\cite{BauschkeGulerLewisSendov}.  Put $h_q$ into normal form relative to $e'$:
\[
 h_q\sim t^3-3\|y\|^2t+2q_{e'}(y),\qquad q_{e'}\in H_4,
\]
The construction may be chosen continuously, so $q_{e'}\to q$, and the
points $p_a$ give contacts $v_a(e')\in Z(q_{e'})$.
For every four-element set $I\subseteq\{1,\ldots,5\}$,
\[
 \{v_a(e')\mid a\in I\}\text{ is linearly dependent}
 \quad\Leftrightarrow\quad
 e'\in\operatorname{span}\{p_a\mid a\in I\}.
\]
Indeed, changing direction projects along $e'$, and the remaining coordinate
changes are invertible.  The five spans on the right are proper hyperplanes
by Lemma~\ref{lem:generic-contacts}; choosing $e'\to e$ outside their union
makes the new contacts poised.  Thus every generic maximizer is a limit of
five-contact cubics.

The maximizer varies continuously at every functional having a unique
maximizer.  Hence the preceding argument and Lemma~\ref{lem:generic-contacts}
give density among exposed points.  The density of exposed points among
extreme points~\cite{Rockafellar} then gives density
in $\operatorname{Ext}(H_4)$.
\end{proof}

\subsection{Clifford realizations for five-contact cubics}

We use standard quaternionic notation.  Write
$\operatorname{Herm}_r(\mathbb H)$ for the quaternionic Hermitian matrices
and $\det_M$ for the Moore determinant; see~\cite{Aslaksen} for background.
For the spectral theorem and congruence identities, we also use the
account in~\cite[expanded version, Appendix~A]{ScottSokal}.
The Cauchy--Binet formula needed below is the restricted Moore-determinant
formula with real diagonal weights. Sylvester's identity in this setting
follows by comparing the nonzero eigenvalues of $WW^*$ and $W^*W$.  On
$\operatorname{Herm}_2(\mathbb H)$,
\[
 \det_M\begin{pmatrix}\alpha&\bar z\\z&\beta\end{pmatrix}
 =\alpha\beta-|z|^2
\]
has signature $(1,5)$, and every
$B\in\operatorname{Herm}_2(\mathbb H)$ with $\tr B=0$ satisfies
$B^2=\frac12\tr(B^2)I_2$.  Inner products on quaternionic vector spaces
are understood through their real parts.

The multiaffine stable-polynomial part of the next proof uses reciprocal
duality and the one-positive-eigenvalue description of stable quadratics,
as developed in~\cite{ChoeOxleySokalWagner}.  The proof includes these
steps explicitly before constructing the quaternionic pencil.

\begin{proposition}
\label{prop:five-quaternionic}
Every five-contact cubic $q\in H_4$ admits a definite quaternionic
Hermitian determinantal representation
\[
 h_q(t,x)=\det_M(tI_3+A(x)),\qquad
 A\colon\mathbb R^4\to\operatorname{Herm}_3(\mathbb H)
 \text{ linear}.
\]
\end{proposition}

\begin{proof}
Let $(v_a)_{a=1}^5$ be the poised contacts, and put
$p_a=(1,v_a)\in\mathbb R^5$.  Every four of the $p_a$ are independent
because the corresponding four $v_a$ are.

Suppose first that all five $p_a$ are independent.  Let
$\Phi\colon\mathbb R^5\to\mathbb R^5$ be given by $\Phi e_a=p_a$ and put
$f(t)=h_q(\Phi t)$.  Singularity at the coordinate points makes $f$
multiaffine, so $f(t)=\sum_{|S|=3}a_S\prod_{a\in S}t_a$.
The positive orthant is carried into the hyperbolicity cone of $h_q$, and its interior into the interior.
Thus $f$ is hyperbolic in every positive direction, hence
real stable, and is nonnegative on the positive orthant.  Restricting to
each coordinate three-plane gives $a_S\geqslant0$.  Consider the reciprocal
quadratic
\[
 g(t)=-\Bigl(\prod_{a=1}^5t_a\Bigr)
 f(-t_1^{-1},\ldots,-t_5^{-1})
 =\sum_{a<b}\beta_{ab}t_at_b,
\]
which is real stable because $z\mapsto-1/z$ preserves the upper half-plane.
Moreover,
\[
 \beta_{ab}=a_{\{1,\ldots,5\}\setminus\{a,b\}}\geqslant0.
\]
Write $g(t)=t^T\mathsf Bt$ for a real symmetric
matrix $\mathsf B$.  If
$v^T\mathsf B\mathbf1=0$, real-rootedness of
$g(v+s\mathbf1)=v^T\mathsf Bv+s^2g(\mathbf1)$ forces
$v^T\mathsf Bv\leqslant0$.  Since
$\gamma:=g(\mathbf1)>0$, the matrix $\mathsf B$ has exactly one positive
eigenvalue.  Since the Moore determinant on
$\operatorname{Herm}_2(\mathbb H)$ has signature $(1,5)$ and $g$ has
zero diagonal, Sylvester's law gives determinant null matrices
$X_a\in\operatorname{Herm}_2(\mathbb H)$ such that
\[
 g(t)=\det_M\Bigl(\sum_at_aX_a\Bigr).
\]
Denote by
\[
 \langle X,Y\rangle_L
 :=\frac12\bigl(\det_M(X+Y)-\det_M X-\det_M Y\bigr)
\]
the Lorentz bilinear form on $\operatorname{Herm}_2(\mathbb H)$.
Polarizing the representation of $g$ gives
$\langle X_a,X_b\rangle_L=\mathsf B_{ab}$.  Thus, putting
$S=\sum_aX_a$, the nonnegative coefficients of $g$ imply
\[
 \langle S,X_a\rangle_L
 =\sum_b\mathsf B_{ba}
 =\frac12\sum_{b\ne a}\beta_{ab}\geqslant0,
\]
where $\beta_{ba}=\beta_{ab}$ for $a<b$.
Since $\det_M S=\gamma>0$, replacing every
$X_a$ by $-X_a$ if necessary makes $S$ positive definite.  The nullity
of $X_a$ and the preceding inequality then imply $X_a\geqslant0$.
After congruence by $S^{-1/2}$, we obtain
\[
 Y_a=d_ad_a^*,\qquad \sum_aY_a=I_2,\qquad
 \det_M\Bigl(\sum_at_aY_a\Bigr)=\frac{g(t)}\gamma.
\]
Let $W=(d_1\ \cdots\ d_5)\in\operatorname{Mat}_{2,5}(\mathbb H)$, so
$WW^*=I_2$.  Put $E=W^*W$ and $K_0=I_5-E$, and choose
$V\in\operatorname{Mat}_{3,5}(\mathbb H)$ with
$V^*V=K_0$ and $VV^*=I_3$.  If $S_0\subset\{1,\ldots,5\}$ is a triple
with complementary pair $S_0^c=\{a,b\}$, the quaternionic singular-value
decomposition gives
\[
 \begin{aligned}
 \det_M(V_{S_0}V_{S_0}^*)
 &=\det_M(V_{S_0}^*V_{S_0})
  =\det_M((K_0)_{S_0,S_0})\\
 &=\det_M(I_3-W_{S_0}^*W_{S_0})
  =\det_M(I_2-W_{S_0}W_{S_0}^*)\\
 &=\det_M(W_{S_0^c}W_{S_0^c}^*)
  =\frac{\beta_{ab}}\gamma.
 \end{aligned}
\]
The restricted Moore Cauchy--Binet identity
\cite[expanded version, Proposition~A.3(g)]{ScottSokal} therefore yields
\[
 f(t)=\det_M\left(
 \gamma^{1/3}V\operatorname{diag}(t_1,\ldots,t_5)V^*\right).
\]
The pulled-back pencil $L$ satisfies $\det_M L=h_q$ and
$L(\Phi\mathbf1)=\gamma^{1/3}I_3\succ0$.  Since $e=(1,0)$ and
$\Phi\mathbf1$ lie in the same open hyperbolicity cone, constancy of
inertia gives $L(e)\succ0$.  As $\det_M L(e)=h_q(e)=1$, congruence by
$L(e)^{-1/2}$ gives the required representation $tI_3+A(x)$ without
changing its determinant.

Suppose instead that the five $p_a$ are dependent.  Their span is a
hyperplane.  Using $p_1,\ldots,p_4$ as coordinates on it and writing
$p_5=u$, singularity at the coordinate points makes the restriction
$F(t)=\sum_{a=1}^4c_a\prod_{b\ne a}t_b$.
Every $u_i$ is nonzero, since every four $p_a$ are independent.
The gradient equations at $u$, after multiplication by
$u_i/(u_1u_2u_3u_4)$, show that all $c_a/u_a$ are equal and hence zero.
Thus
$h_q$ is divisible by the equation of this hyperplane.  Its $t$-coefficient
is nonzero because the $v_a$ span $\mathbb R^4$, and comparison of
coefficients gives a linear form $\lambda$ such that
\[
 h_q(t,x)=(t-\lambda(x))
 \bigl(t^2+\lambda(x)t+\lambda(x)^2-3\|x\|^2\bigr).
\]
Hyperbolicity of the quadratic factor gives
$\delta(x):=3\|x\|^2-\frac34\lambda(x)^2\geqslant0$.
Choose a real-linear map
$\beta\colon\mathbb R^4\to\mathbb H$ with
$|\beta(x)|^2=\delta(x)$, using $\dim_{\mathbb R}\mathbb H=4$, and set
\[
 A(x)=-\lambda(x)\ \oplus\
 \begin{pmatrix}
  \lambda(x)/2&\beta(x)\\
  \overline{\beta(x)}&\lambda(x)/2
 \end{pmatrix}.
\]
Then
\[
 \det_M(tI_3+A(x))
 =(t-\lambda(x))
 \left(\left(t+\frac12\lambda(x)\right)^2-\delta(x)\right)
 =h_q(t,x).\qedhere
\]
\end{proof}

\begin{remark}
Replacing the quaternionic entries by their complex or real matrix
representations turns the Moore representation above into a definite
complex Hermitian determinantal
representation of $h_q^2$ of size $6$, and a definite real symmetric
representation of $h_q^4$ of size $12$.
Both pencils describe the hyperbolicity cone of $h_q$, so its
spectrahedrality already follows at this stage.  The Clifford realization
constructed below is needed for the subsequent convexity argument.
\end{remark}

We now start turning the Moore determinantal representation into a Clifford realization. The following is the most important technical ingredient.

\begin{proposition}
\label{prop:contact-data}
Let $q\in H_4$ satisfy $q(e_1)=1$, and suppose that $h_q$ admits a
quaternionic Hermitian determinantal representation
\[
 h_q(t,x)=\det_M(tI_3+A(x)),\qquad
 A\colon\mathbb R^4\to\operatorname{Herm}_3(\mathbb H)
 \text{ linear}.
\]
Put $U=e_1^\perp$.  Since $q$ attains its maximum on the unit sphere
at $e_1$, its derivative in every direction of $U$ vanishes there.
Thus the term quadratic in $s$ and linear in $y$ is zero, and we may write
\[
 q(s,y)=s^3+3sy^*Dy+r(y),
 \qquad s\in\mathbb R,\quad y\in U,
\]
where $D=D^*$ and $r$ is cubic.  Let $T$ be the canonical symmetric
pencil of $r$, and put
\[
 P=\frac{I+D}{2},\qquad Q=\frac{I-D}{2}.
\]
Then $Q\geqslant I/4$, and there are a finite-dimensional real Hilbert
space $F$, a unit vector $\xi\in F$, a linear map $a\colon U\to F$, and
a linear map $J\colon U\to\operatorname{Hom}(U,F)$ such that
\[
 \begin{gathered}
 a^*\xi=0,\qquad J(y)^*\xi=Dy,\qquad a^*a=P,\\
 a^*J(y)+J(y)^*a=T(y),\qquad
 J(y)^*J(y)=\|y\|^2Q.
 \end{gathered}
\]
\end{proposition}

\begin{proof}
Comparing coefficients in $\det_M(tI_3+A(x))=h_q(t,x)$ gives
\[
 \tr A(x)=0,\qquad
 \frac12\tr(A(x)^2)=3\|x\|^2,
 \qquad \det_M A(x)=2q(x).
\]
Since $h_q(t,e_1)=(t-1)^2(t+2)$, a quaternionic unitary conjugation
lets us assume $A(e_1)=\operatorname{diag}(2,-1,-1)$.
For $y\in U$, let $\alpha(y)$ and $B(y)$ be the diagonal blocks of
$A(0,y)$.  Taking the trace of $A(0,y)$ and comparing the coefficients
of $s$ in $\frac12\tr(A(s,y)^2)=3(s^2+\|y\|^2)$ give, respectively,
$$\alpha(y)+\tr B(y)=2\alpha(y)-\tr B(y)=0.$$
Thus $\alpha=0$, $\tr B=0$, and, for a linear map $z\colon U\to\mathbb H^2$,
\[
 A(s,y)=
 \begin{pmatrix}
  2s&z(y)^*\\
  z(y)&-sI_2+B(y)
 \end{pmatrix},
 \qquad y\in U.
\]
Since $B(y)\in\operatorname{Herm}_2(\mathbb H)$ is traceless, there is a
self-adjoint $G\geqslant0$ on $U$ such that
$B(y)^2=(y^*Gy)I_2$.  Taking the trace of the square and expanding the
Moore determinant by the lower-right Schur complement give the following
identities.  The inverse simplifies because $B(y)^2$ is scalar, and
continuity covers singular blocks:
\[
 \|z(y)\|^2+y^*Gy=3\|y\|^2,
 \qquad
 q(s,y)=s^3+\frac{s}{2}
 \bigl(\|z(y)\|^2-2y^*Gy\bigr)
 +\frac12z(y)^*B(y)z(y).
\]
Comparing this with the prescribed form of $q$ and extracting mixed terms gives
\begin{equation}\label{eq:defect-identities}
 \begin{aligned}
 G&=I-2D,\\
 B(y)B(u)+B(u)B(y)&=2(y^*Gu)I_2,\\
 \operatorname{Re}\bigl(z(u)^*z(v)\bigr)
   &=u^*(3I-G)v,\\
 r(y)&=\frac12z(y)^*B(y)z(y).
 \end{aligned}
\end{equation}
The third identity gives $G\leqslant3I$.  Thus
$Q=(I+G)/4\geqslant I/4$ and $P=(3I-G)/4$.
Define the real-bilinear maps $L,\mathcal A\colon U\times U\to\mathbb H^2$ by
\[
 \begin{aligned}
 L(y,u)&=\frac16B(y)z(u)+\frac13B(u)z(y),\\
 \mathcal A(y,u)&=B(y)z(u)-B(u)z(y).
 \end{aligned}
\]
In orthonormal coordinates on $U$, put
\[
 w(y,u)=(y_1u_2-y_2u_1,\ y_1u_3-y_3u_1,\ y_2u_3-y_3u_2)^T.
\]
Coordinate expansion, using the alternation of $\mathcal A$, shows that
the right-hand side below is quadratic in these three minors appearing in $w$.  Define
the real symmetric matrix $K$ by
\begin{equation}\label{eq:K-coordinate}
 w(y,u)^*Kw(y,u)=\frac1{18}\|\mathcal A(y,u)\|^2
 +(y^*Dy)(u^*Du)-(y^*Du)^2.
\end{equation}
Every unit vector in $\mathbb R^3$ is $w(u_1,u_2)$ for an orthonormal
pair.  A plane rotation preserves this vector and makes $u_1^*Gu_2=0$.
Put $g_i=u_i^*Gu_i\in[0,3]$ and
\[
 X=B(u_1)z(u_2),\qquad Y=B(u_2)z(u_1),\qquad
 \alpha=\sqrt{(3-g_1)(3-g_2)},\quad \beta=2\sqrt{g_1g_2}.
\]
By \eqref{eq:defect-identities},
$\|X\|\|Y\|=\alpha\beta/2$.  Expanding
\eqref{eq:K-coordinate} and applying Cauchy--Schwarz therefore gives
\[
 \begin{aligned}
 36\,w(u_1,u_2)^*Kw(u_1,u_2)
 &=\alpha^2+\beta^2-4\operatorname{Re}(X^*Y)\\
 &\geqslant\alpha^2+\beta^2-2\alpha\beta
 =(\alpha-\beta)^2\geqslant0.
 \end{aligned}
\]
Thus $K\geqslant0$.  Take the real Hilbert space
$F=\mathbb R\oplus\mathbb H^2\oplus\mathbb R^3$, where the quaternionic
summand carries the real inner product
$\langle p,q\rangle=\operatorname{Re}(p^*q)$.
The adjoints of $a$ and $J(y)$ below are taken with respect to these real
Hilbert space structures.  Set $\xi=(1,0,0)$ and define
\[
 \begin{aligned}
 a(u)&=(0,\tfrac12z(u),0),\\
 J(y)u&=(y^*Du,L(y,u),K^{1/2}w(y,u)).
 \end{aligned}
\]
The identities $a^*\xi=0$ and $J(y)^*\xi=Dy$ are immediate, while
$a^*a=P$ follows from
$\operatorname{Re}(z(u)^*z(v))=4u^*Pv$ in \eqref{eq:defect-identities}.
Extracting mixed terms from $r$ gives
\[
 \frac12\operatorname{Re}\bigl(z(u)^*L(y,v)
 +L(y,u)^*z(v)\bigr)
 =u^*T(y)v.
\]
Hence $a^*J(y)+J(y)^*a=T(y)$.  Finally, cancellation of the mixed terms
and \eqref{eq:defect-identities} give
\[
 \begin{aligned}
 \|L(y,u)\|^2+\frac1{18}\|\mathcal A(y,u)\|^2
 &=\frac1{12}\|B(y)z(u)\|^2+\frac16\|B(u)z(y)\|^2\\
 &=\|y\|^2u^*Qu-(y^*Dy)(u^*Du).
 \end{aligned}
\]
Together with \eqref{eq:K-coordinate}, this gives
$\|J(y)u\|^2=\|y\|^2u^*Qu$ for every $u$, hence
$J(y)^*J(y)=\|y\|^2Q$.
\end{proof}

\begin{lemma}\label{lem:three-generator}
Let $Y,E$ be three-dimensional real Hilbert spaces, let $F$ be a
finite-dimensional real Hilbert space, and let
$J\colon Y\to\operatorname{Hom}(E,F)$ be linear with
$J(y)^*J(y)=\|y\|^2I_E$.  After complexification there are Hilbert spaces
$K_-,K_+$, isometries $\iota_-\colon E_{\mathbb C}\to K_-$ and
$\iota_+\colon F_{\mathbb C}\to K_+$, and a real-linear map
$C\colon Y\to\operatorname{Hom}(K_-,K_+)$ such that
\[
 \begin{aligned}
 C(y)^*C(y)&=\|y\|^2I_{K_-},&
 C(y)C(y)^*&=\|y\|^2I_{K_+},\\
 C(y)\iota_-u&=\iota_+J(y)u.
 \end{aligned}
\]
\end{lemma}

\begin{proof}
Choose an orthonormal basis $e_1,e_2,e_3$ of $Y$, put $J_i=J(e_i)$, and
set $A_1=J_2^*J_3$, $A_2=J_3^*J_1$, $A_3=J_1^*J_2$.  The $A_i$ are easily seen to be skew-adjoint, so 
each $A_i$ is represented by a real skew-symmetric $3\times3$ matrix.
A singular-value decomposition of their three independent
coefficients, using orientation-preserving orthogonal changes of basis
in $E$ and $Y$, gives orthonormal bases in which the following normal
form holds.  We denote the new basis of $E$ by $f_1,f_2,f_3$, reuse
$e_1,e_2,e_3$ for the new basis of $Y$, and redefine $J_i=J(e_i)$ and
the cyclic products $A_i$ accordingly:
\[
 A_1=\begin{pmatrix}0&0&0\\0&0&-a\\0&a&0\end{pmatrix},\qquad
 A_2=\begin{pmatrix}0&0&b\\0&0&0\\-b&0&0\end{pmatrix},\qquad
 A_3=\begin{pmatrix}0&-c&0\\c&0&0\\0&0&0\end{pmatrix}.
\]
Here $a,b,c\in\mathbb R$ may have either sign, so both changes of basis
can be chosen orientation-preserving.
For the Pauli matrices define, using these new $A_i$,
\[
 \omega\colon\operatorname{Mat}_2(\mathbb C)\to
 \operatorname{End}(E_{\mathbb C}),\qquad
 \omega(I_2)=I_E,\quad \omega(\sigma_k)=-iA_k.
\]
The positive Gram matrix $[J_i^*J_j]$ then has the principal block
\[
 H=\begin{pmatrix}1&-c&-b\\-c&1&-a\\-b&-a&1\end{pmatrix}\geqslant0.
\]
In these bases the Choi matrix of $\omega$ is unitarily congruent to
$\frac12H\oplus\frac12H$, so the complete-positivity criterion in~\cite{Choi} makes
$\omega$ completely positive.

By the dilation theorem for completely positive maps~\cite{Stinespring}, write
$\omega(X)=V^*(X\otimes I_{\mathcal L})V$, where $\mathcal L$ is
finite-dimensional and
$V\colon E_{\mathbb C}\to K^0:=\mathbb C^2\otimes\mathcal L$
is an isometry.  Define
\[
 C_1^0=I_2\otimes I,\qquad
 C_2^0=i\sigma_3\otimes I,\qquad
 C_3^0=-i\sigma_2\otimes I.
\]
The definitions give $V^*(C_i^0)^*C_j^0V=J_i^*J_j$.
Hence the correspondence $\sum_iC_i^0Vu_i\mapsto\sum_iJ_iu_i$ is a
well-defined surjective isometry from
$S:=\sum_iC_i^0V(E_{\mathbb C})\subseteq K^0$ onto
$T:=\sum_iJ_i(E_{\mathbb C})\subseteq F_{\mathbb C}$.
Enlarge $\mathcal L$ if necessary, extending $V$ by zero and using the
same Pauli formulas for $C_i^0$ on the added summands, so that
$\dim K^0\geqslant\dim F_{\mathbb C}$.  Choose a complex Hilbert space
$L$ of dimension $\dim K^0-\dim F_{\mathbb C}$ and identify $T$ with
$T\oplus\{0\}$ in $F_{\mathbb C}\oplus L$.
The orthogonal complements of $S$ and $T\oplus\{0\}$ now have equal
dimensions, so the isometry extends to a unitary
\[
 U_0\colon K^0\to F_{\mathbb C}\oplus L.
\]
Set $K_-=K^0$, $K_+=F_{\mathbb C}\oplus L$, $C_i=U_0C_i^0$,
$\iota_-=V$, and let $\iota_+$ be the natural inclusion.  Then
$C(y)=\sum_i\langle y,e_i\rangle C_i$ has the required properties.
\end{proof}

We now obtain the required Clifford realization.

\begin{proposition}
\label{prop:five-clifford}
Every five-contact cubic in $H_4$ belongs to $R_4$.
\end{proposition}

\begin{proof}
By the $O(4)$-invariance of $R_4$, we may move one contact in a poised
five-tuple to $e_1$.  Proposition~\ref{prop:five-quaternionic} supplies
a representation $h_q(t,x)=\det_M(tI_3+A(x))$, so
Proposition~\ref{prop:contact-data} applies.  Use its notation and maps.
Since $Q\geqslant I/4$, the map
$\widehat J(y)=J(y)Q^{-1/2}$ satisfies the hypothesis of
Lemma~\ref{lem:three-generator}.  Apply that lemma, identify the
common codomain of the maps $J(y)$ with its image in $K_+$, and put
\[
 b=\iota_-Q^{1/2}\colon U_{\mathbb C}\to K_-.
\]
Then
\[
 b^*b=Q,
 \qquad C(y)b(u)=J(y)u.
\]
On $K_+\oplus K_-$ define
\[
 \Gamma_0=\begin{pmatrix}I&0\\0&-I\end{pmatrix},
 \qquad
 \Gamma(y)=\begin{pmatrix}0&C(y)\\C(y)^*&0\end{pmatrix}.
\]
The lemma gives $\Gamma(y)^2=\|y\|^2I$, while $\Gamma_0$ anticommutes
with every $\Gamma(y)$.  Thus $(s,y)\mapsto s\Gamma_0+\Gamma(y)$ defines
a representation $\pi\colon C_4\to\mathcal B(K_+\oplus K_-)$.  Define
\[
 V\colon\mathbb C\oplus U_{\mathbb C}\to K_+\oplus K_-,
 \qquad
 V(s,u)=(s\xi+a(u),b(u)).
\]
Using the identities of Proposition~\ref{prop:contact-data} and $P+Q=I$
gives $V^*V=I_4$ and
\[
 V^*(s\Gamma_0+\Gamma(y))V
 =\begin{pmatrix}s&(Dy)^T\\Dy&sD+T(y)\end{pmatrix}
 =T_q(s,y).
\]
Consequently $\Theta(c)=V^*\pi(c)V$ is ucp and
\[
 x^*\Theta(c(x))x=x^*T_q(x)x=q(x),
\]
so $q\in R_4$.
\end{proof}

\subsection{The five-variable result}

We can now put everything together to obtain our main result.

\begin{theorem}
\label{thm:R4H4}
Every cubic in $H_4$ is Clifford realizable:
\[
 R_4=H_4.
\]
Consequently every hyperbolicity cone of a cubic in five variables is
spectrahedral.
\end{theorem}

\begin{proof}
By Propositions~\ref{prop:density-five}
and~\ref{prop:five-clifford}, a dense subset of
$\operatorname{Ext}(H_4)$ lies in $R_4$.  Proposition~\ref{p:clgeom} says
that $R_4$ is compact, hence closed, and convex.  
Since a compact convex set in finite dimensions is the convex hull of
its extreme points~\cite{Rockafellar}, we obtain $R_4=H_4$.
The spectrahedrality assertion follows from
Theorem~\ref{t:ucp}.
\end{proof}

\begin{remark}
Together with \Cref{p:clgeom}, the equality $R_4=H_4$ shows that
$H_4$ is a spectrahedral shadow. This extends Hildebrand's
semidefinite representability result for $H_3$~\cite{Hildebrand}
to cubics on $\R^4$. The dimension bound is sharp: by
\Cref{t:strict},  $H_m$ for every
$m\geqslant5$ is not a spectrahedral shadow.
\end{remark}

\section{Conclusion}

We have proved spectrahedrality of hyperbolicity cones of cubics in five
variables, but a stronger representation problem remains open:
does every homogeneous cubic $h$ in five variables, hyperbolic with respect
to $e$, admit a representation
\[
 h(z)=\det_M L(z),\qquad L(e)\succ0,
\]
with $L$ a $3\times3$ quaternionic Hermitian linear pencil?
We construct such representations for five-contact cubics, whereas the
extension to all cubics uses convexity of Clifford realizability and does
not establish quaternionic determinantal representations.

It is also natural to ask how far the quaternionic approach can extend
in more variables.  If quaternionic determinantal representations could
always be converted into Clifford realizations in higher dimensions,
then their existence for every hyperbolic cubic in six variables would
imply $R_5=H_5$, contradicting \Cref{t:strict}.  But such a conversion is not proved here.  However, if this obstruction
can be established, a next direction would be to investigate octonionic
analogues and, more general, constructions based on further
Cayley--Dickson algebras.  The appropriate notions of determinant and
positivity, and their implications for hyperbolicity cones, would then
require separate study.

\section*{Acknowledgements}
I would like to thank Andreas Thom for numerous discussions on
operator-algebraic approaches to the generalized Lax conjecture. I would also like to thank Mario Kummer for his
insightful questions and comments on a previous version of this paper,
which greatly helped to improve the exposition.

\section*{Declaration of AI assistance}
\markboth{Declaration of AI assistance}{}
Generative AI tools were used to assist with drafting and revising the
text, checking mathematical arguments, streamlining proofs, and conducting
literature searches.  The author remains
responsible for the mathematical correctness and final content of the
manuscript.

\end{document}